\documentclass[12pt]{article}

\usepackage{amsthm,amsmath,stmaryrd,bbm,hyperref,geometry,color}
\usepackage{amssymb}
\usepackage[english]{babel}
\usepackage[utf8]{inputenc}
\usepackage{graphicx}
\usepackage{verbatim}
\usepackage{enumitem}
\usepackage{etoolbox}
\usepackage{hyperref}

 \makeatletter
\newcounter{author}
\renewcommand*\author[1]{%
  \stepcounter{author}%
  \ifnum\c@author=1
    \gdef\@author{#1}%
  \else
    \xdef\@author{\unexpanded\expandafter{\@author\and#1}}%
  \fi
  \csgdef{author@\the\c@author}{#1}}
\newcommand*\email[1]{%
  \csgdef{email@\the\c@author}{#1}}
\newcommand*\address[1]{%
  \csgdef{address@\the\c@author}{#1}}
\AtEndDocument{%
  \xdef\author@count{\the\c@author}%
  \c@author=1
  \print@authors}
\newcommand*\print@authors{%
  \ifnum\c@author>\author@count
  \else
    \print@author{\the\c@author}%
    \advance\c@author by 1
    \expandafter\print@authors
  \fi}
\newcommand*\print@author[1]{%
  \par\medskip
  
  \begin{tabular}{@{}l@{}}%
    \textsc{\csuse{author@#1}}\\[.25em]
    \begin{minipage}{\textwidth}\csuse{address@#1}
  \end{minipage}\\[.75em]
    \textit{E-mail}:
    \href{mailto:\csuse{email@#1}}{\csuse{email@#1}}
  \end{tabular}
  }
\makeatother
 
\newcommand{\po}{\left(}
\newcommand{\pf}{\right)}
\newcommand{\co}{\left[}
\newcommand{\cf}{\right]}
\newcommand{\cco}{\llbracket}
\newcommand{\ccf}{\rrbracket}
\newcommand{\R}{\mathbb R}

\newcommand{\N}{\mathbb N} 
\newcommand{\dd}{\text{d}}
\newcommand{\na}{\nabla}
\newcommand{\1}{\mathbbm{1}}
\newcommand{\bv}{\mathbf{v}}
\newcommand{\bg}{\mathbf{g}}
\newcommand{\bw}{\mathbf{w}}
\newcommand{\bh}{\mathbf{h}}

\newcommand{\red}[1]{{#1}}

\newtheorem{theorem}{Theorem}
\newtheorem{proposition}{Proposition}
\newtheorem{remark}{Remark}

\newtheorem{corollary}{Corollary}

\title{Wasserstein contraction for the stochastic Morris-Lecar neuron model}

\author{Maxime Herda}
\address{Inria, Université de Lille, CNRS, UMR 8524 - Laboratoire Paul Painlevé,\\ F-59000 Lille, France} 
\email{maxime.herda@inria.fr}

\author{Pierre Monmarché}
\address{Sorbonne Université, CNRS, Laboratoire Jacques-Louis Lions, UMR 7598,\\ Laboratoire de Chimie Théorique, UMR 7616, F-75005 Paris, France}
\email{pierre.monmarche@sorbonne-universite.fr}

\author{Benoît Perthame}
\address{Sorbonne Universit{\'e}, CNRS, Universit\'{e} de Paris, Inria, Laboratoire Jacques-Louis Lions, UMR 7598, F-75005 Paris, France} \email{benoit.perthame@sorbonne-universite.fr}

\begin{document}

\maketitle

\begin{abstract}
    Neuron models have attracted a lot of attention recently, both in mathematics and neuroscience. We are interested in studying long-time and large-population emerging properties in a simplified toy model. From a mathematical perspective, this amounts to study the long-time behaviour of a degenerate reflected diffusion process. Using coupling arguments, the flow is proven to be a contraction of the Wasserstein distance for long times, which implies the exponential relaxation toward a (non-explicit) unique globally attractive equilibrium distribution. The result is extended to a McKean-Vlasov type  non-linear  variation of the model, when the mean-field interaction is sufficiently small. The ergodicity of the process results from a combination of deterministic contraction properties and local diffusion, the noise being sufficient to drive the system away from non-contractive domains.
    
\medskip


\medskip

    
\end{abstract}

\section{Introduction}
\subsection{The model}

Motivated by mean field representations of neuronal networks, we consider the question of long time behaviour for the following voltage-conductance equations.

Let $a\geqslant 0$, $0<V_L < V_E$, $\gamma,g_L,G_M>0$, $G\in\mathcal C^2 ([V_L , V_E ]; (0,G_M])$, and let $f=f(v,g,t)$ solve
\begin{equation}\label{eq:MorrisLecar}
\left\{\begin{array}{l}
\partial_t f + \partial_v \co \po g_L \po V_L - v \pf + g\po V_E-v\pf \pf f\cf + \gamma \partial_g \co\po G(v) - g\pf f \cf- a^2 \partial_{gg} f = 0 \, ,
\\[5pt]
f(V_E,g,t) = f(V_L,g,t) = 0\, ,
\\[5pt]
 \po G(v) - g\pf f - a \partial_{g} f = 0\qquad \text{at }g=0\,.
\end{array}\right.
\end{equation}
Then $f$ is the density of the process $(v_t,g_t)_{t\geqslant0}$ on $[V_L,V_E]\times\R_+$ solving the SDE
\begin{equation}\label{eq:SDE}
\left\{\begin{array}{rcl}
\dd v_t & = & g_L(V_L-v_t) \dd t + g_t(V_E-v_t)\dd t \, ,\\
\dd g_t  & = & \gamma\po G(v_t)-g_t\pf\dd t + \sqrt{2}a  \dd B_t + \dd L_t\,,
\end{array}\right.
\end{equation}
where $(B_t)_{t\geqslant 0}$ is a standard Brownian motion on $\mathbb{R}$ and $(L_t)_{t\geqslant 0}$ is the local time of $g_t$ at $0$, which means $g_t$ is reflected at $0$. We refer to \cite[Chapter 6]{RevuzYor} for details on local times. In particular, $L$ is non-decreasing with a finite variation on any finite time interval, and  the Radon measure $\dd L_t$ on $\R_+$ is almost surely carried by the set $\{t\geqslant 0,\ g_t = 0\}$. The reflection of $g_t$ at $0$ is the stochastic counterpart of the zero-flux boundary condition at the level of the Fokker-Planck equation, given by the last equation of \eqref{eq:MorrisLecar}.

 The {\em kinetic theory} for neuronal networks, has been developed recently as the mean field representation of large networks with applications to the visual cortex~\cite{CTRM06,RKC}. Macroscopic models of neural networks, as the Integrate-and-Fire equation \cite{CCP_JMN11,BoFaTa,delarue2015global}, describe densities of neurons with potential $v$, this means that the conductance $g$ is supposed to be at equilibrium. System~\eqref{eq:MorrisLecar} is a more detailed kinetic representation and the macroscopic limit is formally obtained as $\gamma \to \infty$. 

\red{Equation \eqref{eq:MorrisLecar} is a simple version of those used to represent electric cells with their voltage~$v$ and  conductance $g$ according to the general theory of spiking neurons~\cite{BressloffBook}. More precisely it is related to Morris-Lecar's  formalism~\cite{MorrisLecar} (a variant of Hodgkin-Huxley) which, for a single neuron, reads in full generality
\[
\frac{dv(t)}{dt}  = \sum_{i=1}^I g_i(t) \big(V_i- v(t)\big), \quad \frac{dg_i(t)}{dt}  = \lambda_i \big(v(t)\big) \big(G_I(v(t))- g_i(t)\big), \quad i=1,...,I,
\]
with $i$ the ionic channels, $V_i$ the reversal potentials, $\lambda_i (v)$ the gating frequency and $G_I(v)$ the equilibrium conductance (for instance of the form $G_I(v) = \alpha (1+\tanh(\beta(v-\gamma))$ for some parameters $\alpha,\beta,\gamma>0$, see \cite{MorrisLecar}). Notice that, if it initially holds, we have for all $t$, $\min_i V_i \leq v(t) \leq \max_i V_i$ and thus there are no neuron with the extreme potentials, which explains the boundary condition in $v$. When considering an assembly of neurons, noise becomes important and it is represented by the Brownian term. It acts only on the variable $g$, i.e., the synaptic activity or the opening and closing of ion channels~\cite{CTSL}. For the deterministic version of \eqref{eq:SDE}} i.e., when $a=0$, the stationary points $(v_*,g_*)$ of~\eqref{eq:SDE} are  determined by the relations
\begin{equation}\label{eq:stationary_point}
  V(g):=\frac{g_L V_L + g V_E}{g_L+g},  \qquad  g_*=G\big( V(g_*)\big), \qquad  v_*= V(g_*) .  
\end{equation}
When $G$ decreases such an equilibrium is unique. For more general functions $G$, there can be several; then a uniqueness condition is that $\frac{\dd}{\dd g} G\big( V(g)\big)<1$, which is also written
\begin{equation}\label{eq:condition_stst}
G'(V(g))(V_E-V(g))<(g_L+g).
\end{equation}
This condition also implies the linear stability of this unique equilibrium.
\\

In order to investigate the large time behavior of solutions to \eqref{eq:MorrisLecar}, we use a mathematical approach based on the coupling of stochastic processes as those in~\eqref{eq:SDE}. It has two advantages compared to other deterministic methods used to investigate the long time behavior of kinetic Fokker-Planck equations, such as hypocoercivity methods \cite{villani2006hypocoercivity, dolbeault2015hypocoercivity}. On the first hand, in comparison to \cite{dolbeault2015hypocoercivity}, we never use the steady state in our computations and rather use an \emph{a posteriori} contraction argument to show its existence and uniqueness. This type of methodology involving non-explicit steady states also appears with long-time analyses based on Doeblin’s or Harris's theorem (see for instance \cite{evans2021asymptotic,SaSm22}). On the second-hand, to the expense of weak topology, our analysis does not rely on  the use of specific commutators and regularity estimates as in \cite{villani2006hypocoercivity}.

 Two cases are distinguished in our work: when $\|G'\|_\infty$ is small (which implies uniqueness of the equilibrium for the deterministic case $a=0$) the deterministic drift of \eqref{eq:SDE} is contracting, namely the distance between two solutions of \eqref{eq:SDE} driven by the same Brownian motion almost surely decays, independently from the noise intensity $a$. However, in the general case where $\|G'\|_\infty$ is not assumed small, this deterministic contraction is no longer true and we need to make use of the noise to get two different trajectories closer on average. The difficulty is then that the diffusion matrix is degenerate. This issue appears in \cite{EberleGuillinZimmer} for the kinetic Fokker-Planck diffusion and \cite{ColombaniLeBris} for a stochastic FitzHugh-Nagumo model. Our work is similar to those, although the coupling construction (and a suitable associated modified distance) should be tailored to the specific dynamics at hand, and we have to deal with the reflection boundary.

\medskip

The rest of the paper is organized as follows. The main results are stated in Section~\ref{sec:results} and proven in Section~\ref{sec:proofs}.  They are extended to systems modeling interacting neurons in Section~\ref{sec:extensions}.

\subsection{Main results}\label{sec:results}

In the following, we state our result for the stochastic system \eqref{eq:SDE}. Denote by $(P_t)_{t\geqslant 0}$ the semigroup associated to \eqref{eq:SDE}, namely
\[\int_{[V_L,V_E]\times\R_+} P_t \varphi f_0 = \mathbb E_{f_0}\po \varphi(v_t,g_t)\pf = \int_{[V_E,V_L]\times\R_+} \varphi f_t\]
for all bounded measurable test function $\varphi$ and initial distribution $f_0$. In other words, $P_t$ is the (backward) solution operator for the dual of~\eqref{eq:MorrisLecar}.

\begin{theorem}[Deterministic contraction]\label{thm:synchrone}
Let $(v_t,g_t)_{t\geqslant 0}$ and $(v_t',g_t')_{t\geqslant 0}$ be two solutions of \eqref{eq:SDE} driven by the same Brownian motion $B_t$. Assume that 
\begin{equation}\label{eq:condition_synchrone}
    (V_E-V_L)\|G'\|_\infty < g_L\,.
\end{equation}
Then, there are constants $C,\lambda>0$ (explicit in the proof and independent from the noise intensity~$a$) such that almost surely, for all $t\geqslant 0$,
\begin{equation}\label{eq:contract}
|(v_t,g_t) - (v_t',g_t')|  \leqslant C e^{-\lambda t} |(v_0,g_0)-(v_0',g_0')|\,.
\end{equation}
\end{theorem}

\begin{corollary} \label{cor:ML}
Under the assumption and with the notations of Theorem \ref{thm:synchrone}, for all \red{Lipschitz continuous $\varphi \in \mathcal C^1([V_L,V_E]\times\R_+, \R)$}  and all $t\geqslant 0$,
\begin{equation}\label{eq:gradient}
|\na P_t \varphi| \leqslant C e^{-\lambda t} P_t|\na\varphi|\,.  
\end{equation}
Moreover, there exists a unique invariant distribution $\mu_\infty^{(a)}$ of \eqref{eq:MorrisLecar}, which satisfies the following log-Sobolev inequality:  for all $\varphi\in\mathcal C^\infty([V_L,V_E]\times\R_+)$ with $\int \varphi^2 \mu_\infty^{(a)} = 1$,
\begin{equation}
    \label{eq:LSI}
    \int_{[V_L,V_E]\times\R_+} \varphi^2 \ln \varphi^2 \mu_\infty^{(a)} \leqslant \frac{2C}{\lambda} \int_{[V_L,V_E]\times\R_+}  |\nabla \varphi|^2 \mu_\infty^{(a)}.
\end{equation}

Finally, the deterministic ODE corresponding to $a=0$ admits a unique stationary state $(v_*,g_*)$ and, for all $a>0$,
\[\int_{[V_L,V_E]\times \R_+} \po |v-v_*|^2 + |g-g_*|^2 \pf \mu_\infty^{(a)}(\dd v,\dd g) \leqslant \frac{C}{\lambda} a^2\,.
\]
\end{corollary}

Notice that condition \eqref{eq:condition_synchrone} is stronger than \eqref{eq:condition_stst} and thus, by itself, existence and uniqueness of a steady state for the deterministic system is obvious.

Since the diffusion matrix of the process is degenerate,  the log-Sobolev inequality \eqref{eq:LSI} does not imply the exponential decay of the relative entropy along the flow \eqref{eq:MorrisLecar}. It could however allow to get an hypocoercive decay (for $(P_t)_{t\geqslant 0}$ rather than the dual operator acting on relative densities) using \cite[Theorem 10]{monmarche2019generalized}, which doesn't require the explicit knowledge of the invariant measure. We do not address this question, and refer to \cite{BakryGentilLedoux} for details on the log-Sobolev inequality and some of its consequences. 

The next result deals with the general case where \eqref{eq:condition_synchrone} is not enforced. The estimates are no longer uniform in $a>0$ (and the result doesn't hold for $a=0$). For $p\geqslant 1$, we write $\mathcal P_p([V_L,V_E]\times \R_+)$ the set of probability measures on $[V_L,V_E]\times \R_+$ with a finite $p^{th}$ moment, and $\mathcal W_p$ the $L^p$-Wasserstein distance on $\mathcal P_p([V_L,V_E]\times \R_+)$, defined as
\[\mathcal W_p(\nu,\mu) = \inf_{\pi\in\mathcal C(\nu,\mu)} \po \int_{([V_L,V_E]\times \R_+)^2}|x-x'|^p\pi(\dd x,\dd x')\pf^{1/p}\]
where $\mathcal C(\nu,\mu)$ is the set of probability measures on $([V_L,V_E]\times \R_+)^2$ with marginals $\nu$ and $\mu$.

\begin{theorem}[Noise-induced contraction]\label{thm:miroir}
For any values of the parameters \red{($G$, $\gamma$, $a$, $g_L$, $V_E$, $V_L$)} with $a>0$ there are constants $C,\lambda>0$ (explicitly given in the proof) such that, for any initial distributions $\nu,\mu\in\mathcal P_1([V_L,V_E]\times \R_+)$ and all $t\geqslant 0$,
\begin{equation}
    \label{eq:contractW_1}
    \mathcal W_1(\nu P_t,\mu P_t ) \leqslant C e^{-\lambda t}  \mathcal W_1(\nu ,\mu  )\,.
\end{equation}

Moreover, the process admits a unique invariant measure $\mu_\infty^{(a)}$. For all $p>q\geqslant 1$, $\mu_\infty^{(a)}\in\mathcal P_p([V_L,V_E]\times \R_+)$ and for  all $\nu \in\mathcal P_p([V_L,V_E]\times \R_+)$, $\mathcal W_q(\nu P_t,\mu_\infty^{(a)}) \rightarrow 0$ as $t\rightarrow \infty$. For all $\nu \in\mathcal P([V_L,V_E]\times \R_+)$, $\nu P_t$ \red{weakly} converges to $\mu_\infty^{(a)}$.
\end{theorem}

The Wasserstein contraction \eqref{eq:contractW_1} implies ergodicity, a Markovian Central Limit Theorem and exponential concentration inequalities for empirical averages of the process, see e.g. \cite{Djellout} and references within. \red{Recall that the convergence in the $\mathcal W_q$ sense is equivalent to the weak convergence plus convergence of all moments up to order $q$ \cite{Villani}.} Notice that, in Theorem~\ref{thm:synchrone}, the deterministic contraction~\eqref{eq:contract} straightforwardly implies an exponential Wasserstein contraction similar to \eqref{eq:contractW_1} but for $\mathcal W_p$ distances for all $p\geqslant 1$ \red{(see \cite{M27} for details).}

\section{Proofs}\label{sec:proofs}

\subsection{The coupling(s)}\label{sec:couplings}

In this section we introduce the couplings that will be used in the proofs.

Let $B,B'$ be two independent standard Brownian motions on $\R$, and let $\alpha,\beta$ be two $\mathcal C^2$ Lipschitz non-negative functions on~$\R$ with $\alpha^2 +\beta^2 = 1$. Given some initial condition $(v_0,g_0,v_0',g_0')$, consider $(v_t,g_t,v_t',g_t')$ the solution of
\begin{equation}
    \label{eq:couple_def}\left.
\begin{array}{rcl}
\dd v_t & = & g_L(V_L-v_t) \dd t + g_t(V_E-v_t)\dd t\\
\dd g_t  & = & \gamma\po G(v_t)-g_t\pf\dd t + \sqrt{2}  a\co \alpha_t\dd B_t + \beta_t \dd B_t'\cf  + \dd L_t\\
\dd v_t' & = & g_L(V_L-v_t') \dd t + g_t'(V_E-v_t')\dd t\\
\dd g_t'  & = & \gamma\po G(v_t')-g_t'\pf\dd t + \sqrt{2}  a  \co \alpha_t\dd B_t - \beta_t  \dd B_t'\cf + \dd   L_t'\,,
\end{array}\right\}
\end{equation}
where \[\alpha_t = \alpha(|g_t-g_t'|^2),\quad  \beta_t = \beta(|g_t-g_t'|^2), \]   and $L, L'$ are respectively the local time at $0$ of~$g_t$ and~$g_t'$. 

By \cite[Theorem 5.1]{saisho}, the process is well-defined, at least up to time $\tau_n = \inf\{t\geqslant 0, g_t+g_t'\geqslant n\}$ for all $n\in\N$, and thus for all times by a standard Lyapunov argument (using e.g. that $(v,g,v',g') \mapsto g^2+(g')^2$ is a Lyapunov function).

By Levy's characterization of the Brownian motion, the processes $W,W'$ defined by
\[W_t = \int_0^t \co \alpha_s\dd B_s + \beta_s \dd B_s'\cf\,,\qquad  W_t' = \int_0^t  \co \alpha_s\dd B_s - \beta_s \dd B_s'\cf\]
are Brownian motions too. As a consequence, $Z_t=(v_t,g_t)$ and $Z_t'=(v_t',g_t')$ are solutions of \eqref{eq:SDE} driven by $W$ and $W'$, so that their laws solve the PDE \eqref{eq:MorrisLecar}.

When $\alpha = 1$ (and thus $\beta=0$), the process $(v_t,g_t,v_t',g_t')$ is called a {\em synchronous coupling} for~\eqref{eq:SDE}. When $\beta(s) = \1_{s>0}$, this is called a {\em mirror (or reflection) coupling} for~\eqref{eq:SDE}. Although~$\alpha$ and~$\beta$ are not $\mathcal C^2$ in this case, it has been shown to be well-defined, at least for elliptic diffusion processes, see e.g. \cite{cranston1991gradient,kendall1986nonnegative,wang1994application} for reflected processes. In order to avoid technical issues in our hypoelliptic context, and since we focus on the Wasserstein distance rather than the total variation (so that we do not need coalescent processes), we will work with a regularized version in Section~\ref{sec:miroir}.

In the general case, for any $f_1,f_2\in\mathcal C^2(\R)$,   by Itô's formula,
\begin{eqnarray}
    \lefteqn{\dd f_1\po |v_t-v_t'|^2\pf}\nonumber\\  
 & = & f_1'\po |v_t-v_t'|^2\pf\co-2(g_L+g_t) |v_t-v_t'|^2 + 2(v_t-v_t')(g_t-g_t')(V_E-v_t') \cf \dd t, \label{eq:ito1}
\end{eqnarray}
\red{and}
\begin{eqnarray}
 \lefteqn{\dd f_2\po |g_t-g_t'|^2\pf}\nonumber\\
 & =&  2(g_t-g_t')  f_2'\po |g_t-g_t'|^2\pf   \co \gamma\po G(v_t)-g_t-G(v_t')+g_t'\pf\dd t + 2\sqrt{2}  a\beta_t \dd B_t'   \cf   \nonumber \\
 & \ &+8a^2 \beta_t^2  \po   f_2'\po |g_t-g_t'|^2\pf + 2|g_t-g_t'|^2 f_2''\po |g_t-g_t'|^2\pf \pf \dd t     \nonumber \\
 &  \ &- 2g_t'  f_2'\po |g_t'|^2\pf  \dd L_t -  2 g_t  f_2'\po |g_t|^2\pf  \dd  L_t' \,. \label{eq:ito2}
 \end{eqnarray}
In the last two terms, we used that $L$ (\emph{resp.} $L'$)  is constant whenever $g_t> 0$ (\emph{resp.} $g_t'> 0$). Moreover, since $L$ and $ L'$ are non-decreasing, as soon as $f_2$ is non-decreasing, the contribution of the local times is negative (which is a general fact when the boundary of the domain is convex, see \cite[Lemma 2.1]{wang1994application}).

\subsection{Synchronous coupling: proof of Theorem~\ref{thm:synchrone}}

This section is devoted to the proofs of Theorem~\ref{thm:synchrone} and Corollary~\ref{cor:ML}.

\begin{proof}[Proof of Theorem~\ref{thm:synchrone}]
Taking $\alpha = 1$, $\beta=0$ in the definition of the process \eqref{eq:couple_def} (which is thus a synchronous coupling) and $f_1(x)=f_2(x)=x$ in \eqref{eq:ito1} and \eqref{eq:ito2} we get, for $A>0$, almost surely for all $t\geqslant 0$,
\begin{eqnarray*}
\frac12 \dd \po |v_t-v_t'|^2 + A |g_t-g_t'|^2 \pf   &\leqslant & - g_L  |v_t-v_t'|^2 \dd t + (V_E-V_L)|v_t-v_t'|  |g_t-g_t'| \dd t\\
& & + A \gamma  \po \|G'\|_\infty |v_t-v_t'||g_t-g_t'| -|g_t-g_t'|^2  \pf \dd t \,.
\end{eqnarray*}
This is less than $-\lambda  \po |v_t-v_t'|^2 + A |g_t-g_t'|^2 \pf \dd t$ for some $\lambda>0$ if  $A>0$ is such that 
\[Q(A) := \po V_E-V_L +A\gamma \|G'\|_\infty \pf^2 - 4A\gamma g_L <0\,.\]
This is always possible when $G$ is constant, by taking $A$ large enough. When $\|G'\|_\infty \neq 0$, the minimum of $Q$ is attained at
\[A_0 =  \frac{1}{\gamma \|G'\|_\infty}\po \frac{2 g_L }{\|G'\|_\infty} - (V_E-V_L)\pf \,. \]
So, $Q$ takes negative values on $\R_+$ 
 iff $Q(A_0)<0$,
namely iff
\[(V_E-V_L)\|G'\|_\infty <  g_L\,,\]
that is condition \eqref{eq:condition_synchrone}. Under this condition, we thus get 
\[\frac12 \dd \po |v_t-v_t'|^2 + A |g_t-g_t'|^2 \pf  \leqslant -\lambda \po |v_t-v_t'|^2 + A |g_t-g_t'|^2 \pf \,, \]
with $\lambda =- Q(A_0)>0$, from which conclusion easily follows.

\end{proof}

\begin{proof}[Proof of Corollary \ref{cor:ML}]

The first part  follows from \cite{M27}. More precisely, the proof of \eqref{eq:gradient} from Theorem~\ref{thm:synchrone} is given in the proof of \cite[Theorem 1 (iii)$\Rightarrow$(v)]{M27}, and then the log-Sobolev inequality is \cite[Proposition 3]{M27}. \red{Moreover, \eqref{eq:contract} implies that $P_{t_0}$ is a contraction on the  metric space $(\mathcal P_1([V_L,V_E]\times\R_+),\mathcal W_1)$ (which is complete by \cite[Theorem 6.18]{Villani}) for some $t_0$ large enough. Indeed, given $\nu_0,\nu_0' \in \mathcal P_1([V_L,V_E]\times\R_+)$ and any coupling $\pi_0 \in \mathcal C(\nu_0,\nu_0')$, by   considering initial conditions $(z_0,z_0'):=((v_0,g_0),(v_0',g_0'))$ distributed according to $\pi_0$, we get that $(z_t,z_t'):=((v_t,g_t),(v_t',g_t'))$ (as considered in Theorem~\ref{thm:synchrone}, namely as two solutions of \eqref{eq:SDE} driven by the same Brownian motion) is a coupling of $\nu_0 P_t$ and  $\nu_0'P_t$, from which 
\begin{equation*}
    \mathcal W_1(\nu_0 P_t,\nu_0'P_t) \leqslant \mathbb E\po |z_t-z_t'| \pf \leqslant  C e^{-\lambda t} \mathbb E_{\pi_0}\po |z_0-z_0'| \pf\,,
\end{equation*}
and taking the infimum over $\pi_0$ yields the $\mathcal W_1$ contraction
\begin{equation}
    \label{eq:Thm1toThm2}
    \mathcal W_1(\nu_0 P_t,\nu_0'P_t) \leqslant   C e^{-\lambda t}  \mathcal W_1(\nu_0 ,\nu_0')\,.
\end{equation}
Denoting by $\mu_\infty^{(a)}$ the fixed point of $P_{t_0}$  and using that $P_{t_0}P_s = P_s P_{t_0}$ for all $s\geqslant 0$, we get that $\mu_\infty^{(a)}P_s$ is a fixed point of $P_{t_0}$ and thus, by uniqueness of the latter, that $\mu_\infty^{(a)}P_s = \mu_\infty^{(a)} $  for all $s\geqslant 0$, i.e. $\mu_\infty^{(a)}$ is an invariant measure of the semi-group. }

For the comparison with the deterministic case, we follow the  computations of the proof of Theorem~\ref{thm:synchrone} except that, now, $(v_t',g_t')_{t\geqslant 0}$ is associated to $a=0$ (i.e. $a=0$ in the fourth line of \eqref{eq:couple_def}). Then, considering $A$ and $\lambda$ as in the previous section (in particular they are independent from the noise intensity $a$), we find
 \begin{eqnarray*}
\frac12 \frac{\dd}{\dd t} \mathbb E\po  |v_t-v_t'|^2 + A |g_t-g_t'|^2 \pf   &\leqslant & \mathbb E \Big[  - g_L  |v_t-v_t'|^2   + (V_E-V_L)|v_t-v_t'|  |g_t-g_t'| \\
& & + A \gamma  \po \|G'\|_\infty |v_t-v_t'||g_t-v_t'| -|g_t-g_t'|^2  \pf  \Big] +   Aa^2 \\
&\leqslant&  - \lambda   \mathbb E\po  |v_t-v_t'|^2 + A |g_t-g_t'|^2 \pf +   Aa^2 \,. 
\end{eqnarray*}
Notice that Theorem~\ref{thm:synchrone} applies to the case $a=0$, so we find that the condition \eqref{eq:condition_synchrone} implies that the stationary points $(v_*,g_*)$ given by \eqref{eq:stationary_point} are globally attractive for the deterministic ODE.  Furthermore, letting $t\rightarrow \infty$ in the previous bound with $(v'_t,g'_t)=(v_*,g_*)$, we end up with
\[ \int_{[V_L,V_E]\times \R_+} \po |v-v_*|^2 + A |g-g_*|^2 \pf \mu_\infty^{(a)}(\dd v,\dd g) \leqslant \frac{A}{\lambda} a^2\,,\]
where $\mu_\infty^{(a)}$ is the invariant measure with noise intensity $a$.
\end{proof}

\subsection{Mirror coupling: proof of Theorem~\ref{thm:miroir}}\label{sec:miroir}

This section is devoted to the proof of Theorem~\ref{thm:miroir}. To this end, we consider the full coupling introduced in Section~\ref{sec:couplings}. For some $\xi\in(0,1]$, which will be sent to $0$ at the end, we choose  $\alpha,\beta$ such that $\beta(s^2) = 1$ for $s\geqslant \xi$ and $\beta(s^2)=0$ for $s\leqslant \xi/2$ (recall that $\alpha^2+\beta^2=1$). As discussed in Section~\ref{sec:couplings}, this is a regularized version of  a mirror coupling.

\begin{remark}
In fact the use of a mirror coupling (i.e. $\beta(s^2)=1$ for $s$ away from zero) is not crucial in the proof. The important point is to keep some noise in the evolution of the difference between the two processes, i.e. to avoid the synchronous coupling ($\beta=0$). For instance an independent coupling (corresponding to $\alpha(s^2)=\beta(s^2)=1/\sqrt2$ for $s\geqslant \xi$) would also work. The mirror coupling is a natural choice as it maximizes the variance of the noise.
\end{remark}

We write $x_t=|v_t-v_t'|$ and $y_t=|g_t-g_t'|$. \red{To prove \eqref{eq:contractW_1},} we assume that $\|G'\|_\infty \neq 0$ (otherwise \eqref{eq:contractW_1} is simply a consequence of  Theorem~\ref{thm:synchrone}\red{, as we saw with \eqref{eq:Thm1toThm2}}). 

We now have to design a suitable distance between the two processes, that will tend to decrease on average. To do so, let $\theta_1,\theta_2\in \mathcal C^2(\R_+,\R_+)$ be non-decreasing functions such that $\theta_1(r)=\theta_2(r)=0$ for all $r\in[0,\xi/4]$. In particular, $r\mapsto f_i(r)=\theta_i(\sqrt{r})$  is $\mathcal C^2$ for $i=1,2$ and Itô's formula (\eqref{eq:ito1} and \eqref{eq:ito2}) gives
\begin{eqnarray*}
\dd \theta_1(x_t) & \leqslant & \theta_1'(x_t) \co-(g_L+g_t) x_t + y_t(V_E-v_t') \cf \dd t\\
\dd \theta_2(y_t)& = &    \theta_2'(y_t)    \co \gamma\po G(v_t)-G(v_t')-y_t\pf\dd t + 2\sqrt{2}  a\beta_t \dd B_t'   \cf   \\
& & +4a^2 \beta_t^2 \theta_2''(y_t) \dd t  -   \theta_2'(y_t)  (\dd L_t  +  \dd  L_t') \,. 
\end{eqnarray*}
\red{For the last term, we have simply used that, on $\{t\geqslant 0,\ g_t=0\}$ (outside of which, almost surely, $\dd L_t=0$), it holds $2g_t'  f_2'\po |g_t'|^2\pf=2|g_t'-g_t|  f_2'\po |g_t'-g_t|^2\pf  =   \theta_2'(y_t) $, and similarly for the term $2 g_t  f_2'\po |g_t|^2\pf \dd L_t' $. }

Setting $R_t = \theta_1(x_t)+\theta_2(y_t)$ and considering a $\mathcal C^2$ function $\rho :\R_+\rightarrow\R_+$,
\begin{eqnarray*}
\dd \rho( R_t)  &= &  \rho'(R_t) \po \dd \theta_1(x_t) + \dd \theta_2(y_t)\pf + 4 a^2 \beta_t^2 \po \theta_2'(y_t)\pf^2 \rho''(R_t)  \dd t\,.
\end{eqnarray*}
As a consequence, when, additionally,  $\rho,\theta_1$ and $\theta_2$ are  non-decreasing we find
\begin{eqnarray*}
\frac{\dd}{\dd t} \mathbb E\po \rho(R_t)\pf & \leqslant & \mathbb E\po \Psi_t\pf
\end{eqnarray*}
where, using that the contribution of the local times is non-positive, $\Psi_t$ is given by
\begin{eqnarray*}
\Psi_t &= & \rho'(R_t) \co  \theta_1'(x_t) \co- g_L x_t + y_t(V_E-V_L) \cf +  \theta_2'(y_t)     \gamma\po \|G'\|_\infty x_t-y_t\pf  +4a^2 \beta_t^2 \theta_2''(y_t)  \cf \\
& & + 4 a^2 \beta_t^2 \po \theta_2'(y_t)\pf^2 \rho''(R_t)  \,.
\end{eqnarray*}
The key point of the proof is to establish that $\Psi_t \leqslant -\lambda \rho(R_t)$ for some $\lambda>0$  up to some terms that vanish with $\xi$. To that end, we will choose more specific functions $\rho,\theta_1,\theta_2$ (which are represented in Figure~\ref{figure1}). 
\begin{figure}[h!]
    \centering
    \begin{minipage}{0.5\linewidth}
 \includegraphics[scale=0.4]{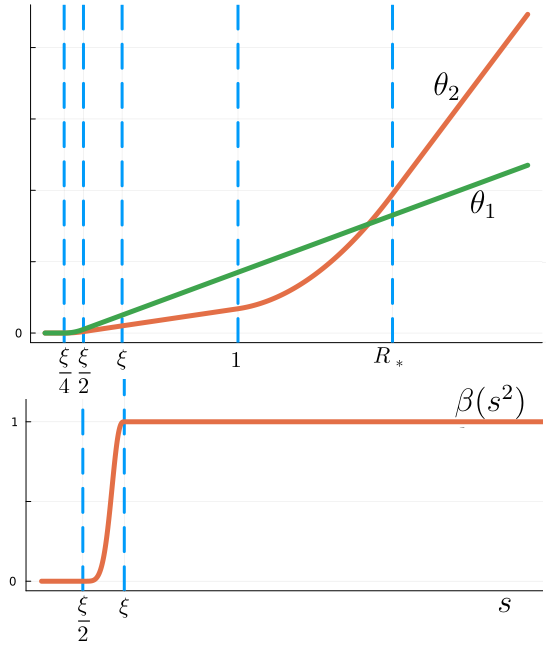}
    \end{minipage}
        \begin{minipage}{0.4\linewidth}
 \includegraphics[scale=0.35]{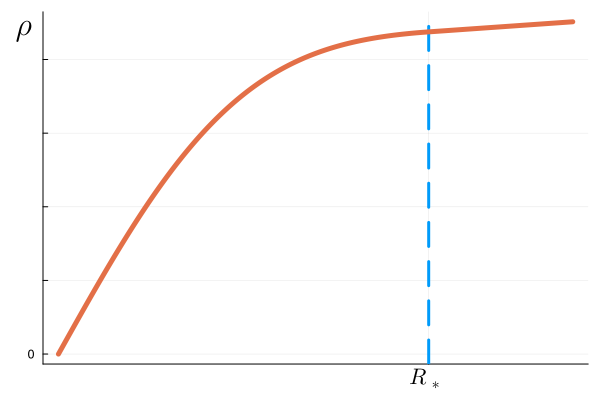}
    \end{minipage}
    \caption{\label{figure1} Left top: graphs of $\theta_1$ (green) and $\theta_2$ (red). Left bottom: graph of $s\mapsto \beta(s^2)$. Right: graph of $\rho$. }
\end{figure}
As will appear in the computations, the important points are the following: the derivatives of $\theta_1$ and $\theta_2$ have to be $0$ at $0$ (for $r\mapsto \theta_i(\sqrt{r})$ to be $\mathcal C^2$). As a consequence, between $\xi/4$ and $\xi/2$, $\theta_2''$ is very large (when $\xi$ is small, since $\theta_2'$ has to go from $0$ to some fixed value $m>0$ in a small interval), and thus it is important that $\beta=0$ in this region (due to the term $\beta_t^2 \theta_2''(y_t)$ in $\Psi_t$). Moreover, the balance between $\theta_1$ and $\theta_2$ has to switch at some point: when $y_t=|g_t-g_t'|$ is too small for the noise to intervene and $x_t=|v_t-v_t'|$ is the main contribution of the distance between the two processes (which can only happen for small distances since $x_t$ is bounded), the term $-g_L \theta_1'(x_t)x_t$ has to control $\theta_2'(y_t) \gamma \|G'\|_\infty x_t$, so that $\theta_2'(y_t)/\theta_1'(x_t)$ has to be sufficiently small. On the contrary, when $y_t$ prevails, the term $-\gamma \theta_2'(y_t) y_t$ has to control $\theta_1'(x_t)y_t(V_E-V_L)$, and thus $\theta_2'(y_t)/\theta_1'(x_t)$ has to be large enough. Finally, in order to exploit the term $\beta_t^2 \rho''(R_t)$ in $\Psi_t$, $\rho$ has to be sufficiently concave -- while still remaining non-decreasing -- up to $R_*$ (after which the contraction is ensured by the deterministic drift, so that $\rho$ can grow linearly, ensuring the equivalence with the standard distance).

Let us now state our precise choices for the functions $\theta_1, \theta_2$ and $\rho$.  Assume that $\theta_1'(r) \in [0,1]$ for all $r\geqslant 0$ and that $\theta_1'(r) = 1$ for all $r\geqslant \xi$. Let
\begin{equation}\label{eq:parametre_g2}
m = \frac{g_L}{2 \gamma  \|G'\|_\infty }\,,\quad M = 4 \frac{V_E-V_L }{\gamma}\,,\quad R_* =\max\po  2 \|G'\|_\infty (V_E-V_L),2+\xi\pf\,.
\end{equation}
Assume that $M>m$ (otherwise Theorem~\ref{thm:synchrone} applies and Theorem~\ref{thm:miroir} follows). Assume that $\theta_2'(r) \in [0,m]$ for $r\in[0,\xi/2]$,  that $\theta_2'(r) = m$ for all $r\in[\xi/2,1]$  (in particular $\theta_2''(r)=0$ on this interval),  that $\theta_2''(r) \in[0,M-m]$  for $r\in[1,R_*]$ and 
$\theta_2'(r) = M$ 
for $r\geqslant R_*$ (so that $\theta_2''(r)=0$ on this interval). Let
\begin{equation}\label{eq:k}
k =   \frac{1}{8a^2 m^2}\po 1+ \frac{1}m\po V_E-V_L +  4a^2 (M-m)  \pf  +  m     \gamma  \|G'\|_\infty\pf \,.  
\end{equation}
 Assume that $\rho(0)=0$, $\rho'(0)=1$, $\rho''(r) = - 2kr \rho'(r)$ for all $r\in[0,R_*]$ and $\rho''(r) = 0$ for all  $r\geqslant R_*+\xi$, with  $\rho''(r) \leqslant 0$ for all $r\geqslant 0$ and $\rho'(r) \in [e^{-k R_*^2 - \xi},1]$ for all $r\geqslant 0$.

First, notice that $\rho(r) \leqslant r$, $\theta_1(r) \leqslant r$ and $\theta_2(r) \leqslant Mr$, so that
\[\rho(R_t) \leqslant R_t \leqslant  x_t+My_t\,.\]
Conversely, $\rho(r) \geqslant e^{-k R_*^2 - \xi}r$, $\theta_1(r) \geqslant r-\xi$ and $\theta_2(r)\geqslant mr - \xi$, so that
\[ e^{k R_*^2 + \xi} \rho(R_t) \geqslant R_t  \geqslant x_t+ m y_t-2\xi \,.\]

Let us prove that, whatever the values of $x_t,y_t\geqslant 0$,
\begin{equation}\label{eq:psi_t_lambda}
\Psi_t \leqslant - \lambda e^{-\xi } \rho(R_t) + C \xi\,,
\end{equation}
for some constants $C,\lambda >0$ independent from $\xi$. We distinguish different cases.

\begin{itemize}
\item When $y_t\leqslant \xi$, we have $\beta_t^2 \theta_2''(y_t)= 0$, $\theta_2'(y_t)\leqslant m$ and $\beta_t^2 \rho''(R_t) \leqslant 0$, hence
\begin{eqnarray*}
\Psi_t &= & \rho'(R_t) \co  \theta_1'(x_t) \co- g_L x_t + y_t(V_E-V_L) \cf +  \theta_2'(y_t)     \gamma\po \|G'\|_\infty x_t-y_t\pf    \cf  \\
&\leqslant & \rho'(R_t) x_t \co   - g_L   \theta_1'(x_t)+  m    \gamma  \|G'\|_\infty    \cf + \xi (V_E-V_L) \,.
\end{eqnarray*}
\begin{itemize}
\item If $x_t\geqslant \xi$, it yields
\begin{eqnarray*}
\Psi_t 
&\leqslant & \rho'(R_t) x_t \co   - g_L   +  m    \gamma  \|G'\|_\infty    \cf + \xi (V_E-V_L) \\
& \leqslant & - \frac{g_L}2 x_t \rho'(R_t) + \xi (V_E-V_L)
\end{eqnarray*}
thanks to the definition \eqref{eq:parametre_g2} of $m$. Since $x_t\geqslant\xi \geqslant y_t$, we get $R_t \leqslant x_t(1+M)$ and
\[\Psi_t \leqslant  - \frac{g_L}{2+2M} R_t \rho'(R_t) + \xi (V_E-V_L) \leqslant  - \frac{g_L}{2+2M} e^{-k R_*^2 - \xi} \rho(R_t) + \xi (V_E-V_L)\,.  \]
\item If  $x_t\leqslant \xi$, it gives
\begin{eqnarray*}
\Psi_t 
&\leqslant & \xi   m    \gamma  \|G'\|_\infty    + \xi (V_E-V_L) \,.
\end{eqnarray*}
Besides, $\rho(R_t) \leqslant (1+M)\xi$ so that, for any $\lambda$ independent from $\xi$,
\[\Psi_t + \lambda \rho(R_t) \leqslant C \xi\]
for some $C>0$ independent from  $\xi$.
\end{itemize}
\item When $y_t\in[\xi,1]$, $\beta_t=1$. Using that $\theta_1'\leqslant 1$, $\theta_2'(y_t)=m$, $\rho''(r) = -2kr \rho'(r)$,
\begin{eqnarray*}
\Psi_t &\leqslant  & \rho'(R_t) \co   y_t(V_E-V_L) +  m     \gamma  \|G'\|_\infty x_t   \cf - 8 k a^2  m^2 R_t \rho'(R_t)  \\
 &\leqslant  & \rho'(R_t)R_t  \co   \frac{V_E-V_L}m +  m     \gamma  \|G'\|_\infty    - 8 k a^2  m^2 \cf   + 2\xi  \po \frac{V_E-V_L}m +  m\gamma  \|G'\|_\infty \pf \\
 & \leqslant & - \rho'(R_t)R_t + C\xi\\
 & \leqslant & - e^{-k R_*^2 - \xi} \rho(R_t) + C\xi
\end{eqnarray*}
for some $C>0$ independent from $\xi$ thanks to the choice \eqref{eq:k} of $k$.
\item When $y_t\in[1,R_*]$, the only difference with the previous case is that $\theta_2''(y_t)=0$. Using that $g''(y_t)\leqslant M-m$, $g'(y_t) \in [m,M]$ and $y_t\geqslant 1$, 
\begin{eqnarray*}
\Psi_t & \leqslant  & \rho'(R_t) \co  y_t(V_E-V_L)  +  M     \gamma  \|G'\|_\infty x_t   +4a^2 (M-m)  \cf + 4 a^2 m^2 \rho''(R_t) \\
 & \leqslant  & \rho'(R_t) \co  y_t\po V_E-V_L +  4a^2 (M-m)\pf   +  M     \gamma  \|G'\|_\infty x_t   \cf + 4 a^2 m^2 \rho''(R_t)  \,.
\end{eqnarray*}
So we are exactly in the same case as before up to a slight change of constant, but again the choice $k$ in \eqref{eq:k} ensures that 
\[\Psi_t \leqslant  - e^{-k R_*^2 - \xi} \rho(R_t) + C\xi\] 
for some $C>0$ independent from $\xi$.
\item When $y_t\geqslant R_*$, we have $\theta_2'(y_t)=M$ hence, using that $\rho''\leqslant 0$  and  $x_t\leqslant V_E-V_L$,
\begin{eqnarray*}
\Psi_t &\leqslant  & \rho'(R_t) \co y_t(V_E-V_L)  +  M    \gamma\po \|G'\|_\infty (V_E-V_L)-y_t\pf  \cf  \\
& \leqslant & - \frac{M \gamma y_t}{4} \rho'(R_t) 
\end{eqnarray*} 
thanks to the choices of $M$ and $R_*$ in \eqref{eq:parametre_g2}. Moreover,
\[y_t \geqslant \frac{y_t+R_*}2 \geqslant \frac{y_t + x_t R_*/(V_E-V_L)}{2} \geqslant \min\po \frac1{2M}, \frac{R_*}{2(V_E-V_L)}\pf  R_t\,.\]
Using again that $R_t \rho'(R_t) \geqslant e^{-k R_*^2 - \xi} \rho(R_t)$, we thus get
\[\Psi_t \leqslant - \frac{\gamma }{8}\min\po 1, \frac{M R_*}{V_E-V_L}\pf   e^{-k R_*^2 - \xi} \rho(R_t)\,.\] 
\end{itemize}
This concludes the proof of \eqref{eq:psi_t_lambda}, with
\[\lambda = \min\po 1,\frac{g_L}{2+2M},\frac{\gamma }{8}, \frac{\gamma M R_*}{8(V_E-V_L)}\pf   e^{-k R_*^2} \,. \]
This implies that 
\[\red{\frac{\dd}{\dd t}}  \mathbb E\po \rho(R_t)\pf \leqslant - \lambda e^{-\xi} \mathbb E\po \rho(R_t)\pf + C\xi\,.\]
Hence,
\begin{eqnarray*}
\mathbb E\po |v_t - v_t'| + m|g_t-g_t'|\pf & \leqslant & \red{e^{kR_*^2+\xi}}\mathbb E \po \rho(R_t)\pf + 2\xi  \\
& \leqslant &  \red{e^{kR_*^2+\xi}}e^{-\lambda e^{-\xi} t}\, \mathbb E \po \red{\rho}(R_0)\pf + \po  \frac{C \red{e^{kR_*^2+2\xi}}}{\lambda}+2\pf \xi \\
&\leqslant  & \red{e^{kR_*^2+\xi}}e^{-\lambda e^{-\xi} t} \,\mathbb E\po |v_0 - w_0'| + M|g_0-g_0'|\pf + \po  \frac{C \red{e^{kR_*^2+2\xi}}}{\lambda}+2\pf \xi\,.
\end{eqnarray*}
Taking initial conditions distributed according to an optimal coupling for  $\mathcal W_1$, we get
\[\mathcal W_1(\nu P_t,\mu P_t) \leqslant \red{C'}\frac{\max(1,M)}{\min(1,m)} e^{-\lambda e^{-\xi} t} \mathcal W_1(\nu,\mu) + C'\xi\]
for some $C'>0$ independent from $\xi\red{\in(0,1]}$. Yet, $\mathcal W_1(\nu P_t,\mu P_t) $ for $t\geqslant 0$ does not depend on  $\xi$ so we let $\xi$ vanish to conclude \red{the proof of \eqref{eq:contractW_1}}.

The existence of an invariant measure is implied by the fact that, for $t$ large enough, $P_t$ is a contraction of $(\mathcal P_1,\mathcal W_1)$, which is a complete space. Hence, for a given $t>0$, $P_t$ has a unique fixed point $\mu_*$ in $\mathcal P_1$, and the semigroup property $P_sP_t=P_{t+s}=P_t P_s$ implies that this fixed point is in fact an invariant measure of $P_s$ for all $s\geqslant 0$. To see that $\mu_*$ has in fact all moments finite, consider for $k\geqslant 2$ the observable $\varphi_k(v,g) = |g|^k$. It is clear that
\[\red{\frac{\dd}{\dd t}} \mathbb E  \po \varphi_k(v_t,g_t) \pf \leqslant - \delta \mathbb E \po \varphi_k(v_t,g_t) \pf  + C\]
for some $C,\delta>0$, from which $\mathbb E_{f_0} \po \varphi_k(v_t,g_t) \pf  \leqslant  e^{-\delta t} \mathbb E_{f_0} \po \varphi_k(v_0,g_0) \pf + C/\delta  $ for all initial condition $f_0$. Letting $t\rightarrow \infty$ gives finite moments for $\mu_*$. To get the long-time convergence in $\mathcal W_q$ for $q>1$, notice that for two random variables $X,Y$, from Hölder's and Minkowski's inequalities, for any $\alpha\in(0,1)$
\begin{eqnarray*}
\mathbb E \po |X-Y|^q \pf &\leqslant &\co \mathbb E \po |X-Y| \pf\cf^{\alpha} \co \mathbb E \po |X-Y|^{\frac{q-\alpha}{1-\alpha}} \pf\cf^{1-\alpha } \\ &\leqslant& \co C_{q,\alpha}\mathbb E \po |X-Y| \pf\cf^{\alpha} \po \co \mathbb E \po |X|^{\frac{q-\alpha}{1-\alpha}} \pf\cf^{1-\alpha } + \co \mathbb E \po |Y|^{\frac{q-\alpha}{1-\alpha}} \pf\cf^{1-\alpha }\pf 
\end{eqnarray*}
for some $C_{q,\alpha}$. From this, in our case, we get that for any probability measures $\nu,\nu'$,
\[\mathcal W_q(\nu,\nu') \leqslant C_{q,\alpha}'\mathcal W_1^{\alpha/q}(\nu,\nu') \po 1+\nu\po \varphi_{(q-\alpha)/(1-\alpha)} \pf + \nu' \po \varphi_{(q-\alpha)/(1-\alpha)} \pf\pf  \]
for some $C_{q,\alpha}'>0$ (we used that $(1-\alpha)/q \leqslant 1$ for simplicity). Noticing that $(q-\alpha)/(1-\alpha)$ may be taken arbitrarily close to $q$ by taking $\alpha$ small enough and using the previous uniform in time moment bounds, we get that for any $\nu\in\mathcal P_p$ for some $p>q$,
\[\mathcal W_q(\nu P_t,\mu_*) \leqslant C e^{-\alpha \lambda t /q }\]
for some $C,\alpha >0$ independent from $t$.

Finally, for any initial condition $\nu$ (not necessarily with moments), we can decompose $\nu = \nu \1_{\mathrm B(n)} + \nu \1_{\mathrm B(n)^c} $ where $\mathrm B(n)$ is the ball of $\R^2$ centered at the origin with radius $n$.  From the previous results, writing $\nu_n = \nu \1_{\mathrm B(n)} / \nu(\mathrm{B}(n)) $ the law of the initial condition conditioned to be in $\mathrm B(n)$, we know that $\nu_n P_t \rightarrow \mu_*$ as $t\rightarrow \infty$, and on the other hand $\nu(\mathrm{B}(n)) \rightarrow 1$ as $n\rightarrow \infty$, which concludes the proof of Theorem~\ref{thm:miroir}.

\subsection{Fokker-Planck interpretation of the coupling}
Although our proofs rely on the probabilistic couplings of Section~\ref{sec:couplings}, let us briefly and informally discuss their  PDE counterpart, that is the Fokker-Planck equation for the law $\Pi_t$ of $(v_t, v_t', g_t, g_t')$. In order to determine the equation on $\Pi_t$, one can use Itô's formula to get, for a test function $\varphi\equiv \varphi(v,v',g,g',t)$ that
\begin{multline*}
\varphi(v_t, v_t', g_t, g_t',t) = \varphi(v_0, v_0', g_0, g_0', 0) + \int_0^t (\partial_s+\mathcal{L}) \varphi(v_s, v_s', g_s, g_s',s) \dd s\\ + 
\int_0^t \partial_g\varphi(v_s, v_s', 0, g_s',s) \dd L_s + 
\int_0^t \partial_{g'}\varphi(v_s, v_s', g_s, 0,s) \dd L_s'\\ + 
\int_0^t \sqrt{2}a\alpha_s(\partial_g+\partial_{g'})\varphi(v_s, v_s', g_s, g_s',s)\dd B_s+ 
\int_0^t \sqrt{2}a\beta_s(\partial_g-\partial_{g'})\varphi(v_s, v_s', g_s, g_s',s)\dd B_s'.
\end{multline*}
with $\mathcal{L}$ the differential operator
\begin{multline*}
\mathcal{L} \varphi = \po g_L \po V_L - v \pf + g\po V_E-v\pf \pf \partial_v \varphi+ \gamma \po G(v) - g\pf  \partial_g\varphi \\
\po g_L \po V_L - v' \pf + g'\po V_E-v'\pf \pf \partial_{v'} \varphi+ \gamma \po G(v') - g'\pf  \partial_{g'}\varphi\\
+ a^2 \partial_{gg} \varphi + 2a^2(\alpha^2(|g-g'|^2)-\beta^2(|g-g'|^2)) \partial_{gg'} \varphi+ a^2 \partial_{g'g'} \varphi.
\end{multline*}
Let us restrict to test functions $\varphi\in C^{1,2}_b([0,T]\times\bar{\Omega})$ with $\Omega =  (V_L,V_E)^2\times(0,\infty)^2$ such that $\partial_g\varphi(g=0) = \partial_{g'}\varphi(g'=0) = 0$, in order to get rid of the local time  contributions. By taking expectations in the formula above one obtains 
\[
\int_\Omega \varphi(T)\dd\Pi_T = \int_\Omega \varphi(0)\dd\Pi_0 + \int_0^T\int_\Omega(\partial_s + \mathcal{L})\varphi(s)\dd\Pi_s.
\]
After formal integration by parts, one recovers the strong form of the Fokker-Planck equation for the coupling
\begin{multline*}
\partial_t\Pi_t + \partial_v \po \co g_L \po V_L - v \pf + g\po V_E-v\pf\cf\Pi_t \pf + \gamma \partial_g\po\co G(v) - g\cf\Pi_t\pf   \\ 
\partial_{v'} \po \co g_L \po V_L - v' \pf + g'\po V_E-v'\pf\cf\Pi_t \pf + \gamma \partial_{g'}\po\co G(v') - g'\cf\Pi_t\pf \\
- a^2 \partial_{gg}\Pi_t - 2a^2 \partial_{gg'}\po\co\alpha^2(|g-g'|^2)-\beta^2(|g-g'|^2)\cf \Pi_t\pf- a^2 \partial_{g'g'} \Pi_t =0,
\end{multline*}
and at the boundaries, on the one hand,
\[
\Pi_t = 0 \text{ when } v = V_E \text{ or } v = V_L \text{ or } v' = V_L \text{ or } v' = V_E,
\]
and, on the other hand,
\[
\begin{array}{lcll}
 \gamma G(v) \Pi_t-a^2\partial_g\Pi_t - 2a^2\partial_{g'}\co\po\alpha^2(|g'|^2)-\beta^2(|g'|^2)\pf\Pi_t\cf&=&0&\text{ when }g = 0,  \\
 \gamma G(v')\Pi_t-a^2\partial_{g'}\Pi_t - 2a^2\partial_{g}\co\po\alpha^2(|g|^2)-\beta^2(|g|^2)\pf\Pi_t\cf&=&0&\text{ when } g' = 0,
\end{array}
\]
and
\[
(\alpha^2(0)-\beta^2(0))\Pi_t = 0 \text{ when } g=g'=0.
\]
Formally, one can check that the marginals follow the original Fokker-Planck equation \eqref{eq:MorrisLecar}. To conclude this section let us stress again that this PDE was not used in our proofs, which are based on the SDE representation of the couplings. The rigorous justification of these computations is unclear because of the combination of boundary reflection and coupling, which can in some cases create density in non-trivial subspaces even when the initial distribution is smooth (see \cite{burdzy2006synchronous} for instance).

\section{Interacting neuron networks}\label{sec:extensions}
\subsection{Finite system of neurons}

For $N>1$, write $\bv=(v^1,\dots,v^N)\in[V_L,V_E]^N$, $\bg=(g^1,\dots,g^N) \in \R_+^N$. Consider the process $(\bv_t,\bg_t)_{t\geqslant 0}$ solving
\begin{equation} \label{eq:interactiong-neuron}
\forall i\in\cco 1,N\ccf\,, \quad \left\{\begin{array}{rcl}
\dd v_t^i & = & g_L(V_L-v_t) \dd t + g_t^i(V_E-v_t^i)\dd t\\
\dd g_t^i  & = & \gamma\po G_i(\bv_t)-g_t^i\pf\dd t + \sqrt{2} a \dd B_t^i + \dd L_t^i\,,
\end{array}\right.
\end{equation}
where $B^1,\dots,B^N$ are independent Brownian motions, $L^i$ is the local time at $0$ of $g^i$ and $G_i \in \mathcal C^1([V_L,V_E]^N,(0,\infty))$. 

\begin{theorem}[Wasserstein contraction for networks]\label{thm:mean_field}
Assume that there exists $K>0$ such that $|\partial_{v^i} G_i(\bv)| \leqslant K$ for all $\bv \in [V_L,V_E]^N$, $i\in\cco 1,N\ccf$. Let $M,m,\lambda,k,R_*$ be as in Section~\ref{sec:miroir} (except that $\|G'\|_\infty$ is replaced by $K$). Then, for all $t\geqslant 0$ and all \red{$\mu,\nu\in\mathcal{P}_1([V_L, V_E]^N\times\mathbb{R}_+^N)$},
\[\mathcal  W_{1,\ell_1}(\nu P_t,\mu P_t) \leqslant \frac{\max(1,M)}{\min(1,m)} e^{-(\lambda-\eta) t} \mathcal  W_{1,\ell_1}(\nu ,\mu )\,,\]
where $\mathcal W_{1,\ell_1}$ is the Wasserstein-$1$ distance associated to $\|(\bv,\bg)-(\bw,\bh)\|_1 = \sum_{i=1}^N |v^i-w^i|+|g^i-h^i|$, and
\[\eta =  M\gamma  e^{kR_*^2 } \max_{1\leqslant i\leqslant N} \sum_{j\neq i }^N  \|\partial_{v^j} G_{i}\|_\infty\,.\]
In particular, if $\lambda>\eta$, the process admits a unique invariant measure $\mu_*^N$, whose moments are all finite and which is globally attractive for $\mathcal W_p$ for all $p\geqslant 1$.
\end{theorem}

\begin{remark}
The point is that when $K$ is independent from $N$, so is $\lambda$.
\end{remark}

\begin{proof}
We consider a particle-wise mirror coupling  as in Section~\ref{sec:miroir}, in the sense that we consider independent Brownian motions $B^i,B'^i$ for all $i\in\cco 1,N\ccf$ and we let $(v^i,g^i,v'^i,g'^i)$ solve \eqref{eq:couple_def}, driven by these Brownian motions, with $\alpha,\beta$ as in Section~\ref{sec:miroir}, except that $G(v)$ is replaced by $G_i(\mathbf{v})$. We also consider $\theta_1,\theta_2,\rho$ some non-negative non-decreasing $\mathcal C^2$ functions and let $x^i_t=|v_t^i - v_t'^i|$, $y^i_t = |g_t^i - g_t'^i|$, $R_t^i = \theta_1(x^i_t) + \theta_2(y^i_t)$. Then, as in Section~\ref{sec:miroir},
\begin{eqnarray*}
\partial_t \mathbb E \po \rho (R_t^i) \pf   & \leqslant  & \mathbb E\po \Psi_t^i\pf
\end{eqnarray*}
with 
\begin{eqnarray*}
\Psi_t^i &= & \rho'(R_t^i) \Bigg{[}  \theta_1'(x^i_t) \co- g_L x^i_t + y^i_t(V_E-V_L) \cf +  \theta_2'(y^i_t)     \gamma\po \sum_{j\neq i }^N  \|\partial_{v^j} G_{j}\|_\infty x^j_t + K x^i_t -y^i_t\pf   \\
& & +4a^2 \beta_{t,i}^2 \theta_2''(y^i_t)  \Bigg{]} + 4 a^2 \beta_{t,i}^2 \po \theta_2'(y^i_t)\pf^2 \rho''(R_t^i) \\
& := & \hat \Psi_t^i + \theta_2'(y^i_t)  \rho'(R_t^i)    \gamma\sum_{j\neq i }^N  \|\partial_{v^j} G_{j}\|_\infty x^j_t\,,
\end{eqnarray*}
where $\hat \Psi_t^i $ is exactly as in Section~\ref{sec:miroir} with $\|G'\|_\infty$ replaced by $K$. So we take $\rho,\theta_1,\theta_2$ exactly as in Section~\ref{sec:miroir} (with $\|G'\|_\infty$ replaced by $K$), in particular $\hat \Psi_t^i  \leqslant -\lambda e^{-\xi} R_t^i +C\xi$ for an arbitrary $\xi$. Besides, we can bound
\[ \theta_2'(y^i_t)  \rho'(R_t^i)    \gamma\sum_{j\neq i }^N  \|\partial_{v^j} G_{j}\|_\infty x^j_t\leqslant M\gamma  e^{kR_*^2 +\xi}  \sum_{j\neq i }^N  \|\partial_{v^j} G_{j}\|_\infty \rho(R_t^j)\,, \]
 and thus,
\[\partial_t \sum_{i=1}^N \mathbb E \po \rho\po R_t^i\pf\pf \leqslant - \po \lambda e^{-\xi}-\eta e^{\xi}\pf \sum_{i=1}^N \mathbb E \po \rho\po R_t^i\pf\pf +C\xi \,.\] 
Conclusion follows as in the case $N=1$.
\end{proof}

\subsection{Non-linear limit in the  mean-field case}

As a simple example of interacting neurons (see \cite{BoFaTa, delarue2015global} for more realistic mean field limits, or \cite{blaustein2023concentration} and references therein for spatially extended mean-field interactions), we consider the case where
\begin{equation}\label{eq:G_iG}
G_i(\bv) = H_0(v^i) + \frac1{N-1} \sum_{j\neq i} H_1(v^i,v^j)
\end{equation}
for some functions $H_0\in \mathcal C^1([V_L,V_E],(0,\infty)),H_1\in\mathcal C^1([V_L,V_E]^2,(0,\infty))$.  In that case\red{, 
\begin{equation}
\label{eq:Gi1}
|\partial_{v^i} G_i(\bv)| \leqslant |H_0'(v^i)| + \frac1{N-1} \sum_{j\neq i} |\partial_{v^i} H_1(v^i,v^j)| \leqslant \|H'_0\|_\infty + \|\nabla H_1\|_\infty\,,
\end{equation}
and, for $j\neq i$,
\begin{equation}
\label{eq:Gi2}
|\partial_{v^j} G_i(\bv)| = \frac1{N-1}  |\partial_{v^j} H_1(v^i,v^j)| \leqslant \frac1{N-1} \|\nabla H_1\|_\infty\,,
\end{equation}
from which $K$ and $\eta$ in Theorem~\ref{thm:mean_field} can be taken independent from $N$.
}

 The law of the first neuron is expected to converge as $N\rightarrow \infty$ to the solution of the non-linear equation
\begin{equation}\label{eq:MorrisLecar_non-linear}
\left\{\begin{array}{l}
\partial_t f + \partial_v \co \po g_L \po V_L - v \pf + g\po V_E-v\pf \pf f\cf \\
\hspace{2cm}+ \gamma \partial_g \co\po H_0(v)+ H_1\star f(v) - g\pf f \cf- a \partial_{gg} f = 0\\
f(V_E,g,t) = f(V_L,g,t) = 0\\
 \po H_0(v)+H_1\star f(v) - g\pf f - a \partial_{g} f = 0\qquad \text{at }g=0\,,
\end{array}\right.
\end{equation}
where
\[H_1\star f(v) = \int_{[V_L,V_E]\times\R_+} H_1(v,w) f(w,h)\dd w\dd h\,.\]
The well-posedness of the system  \eqref{eq:MorrisLecar_non-linear} for $f_0 \in\mathcal P_1([V_L,V_E]\times\R_+)$ follows from a classical fixed-point argument, using again a synchronous coupling to see that the solution is Lipschitz in terms of the non-linearity. We do not detail this and refer to e.g. the similar proof of  \cite[Propostion 1.1]{duong2023reducing}. 

Similarly, propagation of chaos, i.e. convergence of the joint law of a fixed number of neurons of \eqref{eq:interactiong-neuron} toward independent solutions of the non-linear system \eqref{eq:MorrisLecar_non-linear}, is classically obtained using  a synchronous coupling: 

\begin{proposition}[Propagation of chaos]\label{prop:chaos}
Let $a\geqslant 0$, $0<V_L < V_E$, $\gamma,g_L>0$, $H_0\in \mathcal C^1([V_L,V_E],(0,\infty))$, $H_1\in \mathcal C^1([V_L,V_E]^2,(0,\infty))$ 
and $f_0\in\mathcal P_1([V_L,V_E]\times\R_+)$. There exist $C_0,c>0$ such that the following holds. For all integer $N\geq2$,
let $(\mathbf{v},\mathbf{g})$ be a solution of \eqref{eq:interactiong-neuron} in the mean-field case \eqref{eq:G_iG} with initial condition distributed according to $f_0^{\otimes N}$. For  $k\leqslant N$, Denote by $f_t^{(k,N)}$ the law of $(v^1_t,g^1_t,\dots,v^k_t,g^k_t)$. Consider $\bar f$ the solution of \eqref{eq:MorrisLecar_non-linear} with initial condition $f_0$. Then, for all $t\geqslant 0$,
\[\mathcal W_2 \po f_t^{(k,N)}, \bar f_t^{\otimes k}\pf \leqslant \frac{\sqrt{k}}{\sqrt{N}} C_0 e^{ct}\,.\]
\end{proposition}
In fact, using the $\mathcal W_1$ contraction of Theorem~\ref{thm:mean_field}, one can get a similar result uniformly in time, i.e. without the $e^{ct}$ term (with $\mathcal W_1$ instead of $\mathcal W_2$ on the left-hand side). However the  basic propagation of chaos result stated here (in finite time) is already sufficient to allow to let $N\rightarrow \infty$ in Theorem~\ref{thm:mean_field}. Before providing the proof of Proposition~\ref{prop:chaos}, let us thus simply state what this yields (using as in the proof of Theorem~\ref{thm:miroir} that $(\mathcal P_1,\mathcal W_1)$ is a complete space to get a unique fixed point in the case $\lambda>\theta$).

\begin{theorem}[Non-linear Wasserstein contraction]\label{thm:mean_field_NL}
Under the assumptions and with the notations of Theorem~\ref{thm:mean_field} in the mean-field case \eqref{eq:G_iG} (in particular with $\lambda,\red{\eta}$ independent from $N$ \red{thanks to \eqref{eq:Gi1} and \eqref{eq:Gi2}}), for any $f_t,f_t'$ (weak) solutions of the non-linear PDE  \eqref{eq:MorrisLecar_non-linear} with $f_0,f_0'\in\mathcal P_1([V_L,V_E]\times\R_+)$ and all $t\geqslant 0$,
\[\mathcal  W_{1}(f_t,f'_t) \leqslant \frac{\max(1,M)}{\min(1,m)} e^{-(\lambda-\theta) t} \mathcal  W_{1}(f_0,f_0')\,.\]
In particular, if $\lambda>\theta$, then \eqref{eq:MorrisLecar_non-linear} admits a unique stationary solution, which is globally attractive in $\mathcal W_1$.
\end{theorem}

\begin{proof}[Proof of Proposition~\ref{prop:chaos}]
We consider the synchronous coupling of a process $(\mathbf{v},\mathbf{g})$ solving \eqref{eq:interactiong-neuron} and of $N$ independent non-linear processes with law $\bar f_t$, namely, for all $i\in\cco 1,N\ccf$,
\begin{equation*}
\begin{array}{rcl}
\dd v_t^i & = & g_L(V_L-v_t^i) \dd t + g_t^i(V_E-v_t^i)\dd t\\
\dd g_t^i  & = & \gamma\po 
H_0(v_t^i) + \frac1{N-1} \sum_{j\neq i} H_1(v_t^i,v_t^j)
-g_t^i\pf\dd t + \sqrt{2}  a\dd B_t^i + \dd L_t^i\\
\dd v_t'^i & = & g_L(V_L-v_t'^i) \dd t + g_t'^i(V_E-v_t'^i)\dd t\\
\dd g_t'^i  & = & \gamma\po H_0(v_t'^i)+H_1\star \bar f_t(v_t'^i)
-g_t'^i\pf\dd t + \sqrt{2}  a  \dd B_t^i+ \dd  L_t'^i\,,
\end{array}
\end{equation*}
with the same initial conditions distributed according to $f_0^{\otimes N}$. In particular, by Itô calculus, the law of $(v'^i,g'^i)$ solves \eqref{eq:MorrisLecar_non-linear} and is thus equal to $\bar f$, from which
\begin{eqnarray}
\mathcal W_2^2 \po f_t^{(k,N)}, \bar f_t^{\otimes k}\pf & \leqslant &   \mathbb E \po \sum_{i=1}^k |v_t^i-v'^i_t|^2 + |g_t^i-g'^i_t|^2 \pf \nonumber \\
  &=&   \frac kN \mathbb E \po \sum_{i=1}^N |v_t^i-v'^i_t|^2 + |g_t^i-g'^i_t|^2 \pf\,, \label{eq:coupling_propchaos}
\end{eqnarray}
where we used the interchangeability of the neurons.

Similarly to the proof of Theorem~\ref{thm:synchrone}, using in particular that the contribution of the local times is negative and that $H_0,H_1$ are Lipschitz, we get
\begin{eqnarray*}
\frac12 \dd \sum_{j=1}^N \po |v_t^j-v_t'^j|^2 +  |g_t^j-g_t'^j|^2 \pf   &\leqslant & c \sum_{j=1}^N \po |v_t^j-v_t'^j|^2 +  |g_t^j-g_t'^j|^2 \pf \dd t\\
& & + \sum_{j=1}^N \left|H_1\star \bar f_t(v_t'^j) - \frac1{N-1} \sum_{i\neq j} H_1(v_t'^j,v'^i_t)\right|^2 \dd t
\end{eqnarray*}
for some constant $c>0$, where we have bounded
\begin{eqnarray*}
\lefteqn{\left|H_1\star \bar f_t(v_t'^j) - \frac1{N-1} \sum_{i\neq j} H_1(v_t^j,v^i_t)\right|}\\ &\leqslant& \left|H_1\star \bar f_t(v_t'^j) - \frac1{N-1} \sum_{i\neq j} H_1(v_t'^j,v'^i_t)\right| + \frac1{N-1} \sum_{i\neq j} \left| H_1(v_t'^j,v'^i_t)-H_1(v_t^j,v^i_t)\right|  \\
&\leqslant& \left|H_1\star \bar f_t(v_t'^j) - \frac1{N-1} \sum_{i\neq j} H_1(v_t'^j,v'^i_t)\right| + \red{\|H'_1\|_\infty} \po |v_t^j-v'^j_t| + \frac{1}{N-1} \sum_{i\neq j} |v_t^i-v'^i_t|\pf  \,.
\end{eqnarray*}
Using that $v_t'^i$ and $v_t'^j$ are independent when $i\neq j$ and that $H_1\star \bar f_t(w) = \mathbb E[ H_1(w,v'^i_t)]$ for any $w\in[V_L,V_E]$, expanding the square for a fixed $j\in\cco 1,N\ccf$ leads to 
\begin{eqnarray*}
\lefteqn{\mathbb E   \left|H_1\star \bar f_t(v_t'^j) - \frac1{N-1} \sum_{i\neq j} H_1(v_t'^j,v'^i_t)\right|^2}\\
& = & \frac1{(N-1)^2} \sum_{i\neq j} \mathbb E   \left|H_1\star \bar f_t(v_t'^i) - H_1(v_t'^j,v'^i_t)\right|^2   \ \leqslant\   \frac{4\|H_1\|_\infty^2}{N-1}.
\end{eqnarray*}
Using Gronwall's Lemma (and the fact \red{that} the initial conditions are the same for the two processes) thus yields
\[ \sum_{j=1}^N  \mathbb E \po |v_t^j-v_t'^j|^2 +  |g_t^j-g_t'^j|^2 \pf \leqslant \frac{4N}{N-1}\|H_1\|_\infty^2  t e^{ct} \,,\]
 which, thanks to \eqref{eq:coupling_propchaos}, concludes the proof.
 \end{proof}

\section{Conclusion} 

The statistical physics representation of the Morris-Lecar equation is a simplified model for neural assemblies describing the number of neurons with voltage $v$ (a macroscopic variable) and conductance $g$ (a mesoscopic variable). In the deterministic case this system admits various long term behaviours related to the number of stationary states. With a degenerate diffusion on the kinetic variable, we are able to prove exponential convergence to the steady state. Our method uses synchronous or mirror couplings of the related processes. It is flexible and allows us to study the long term behaviour \red{of} $N$ interacting individual neurons as well as their mean field limit.
\\

Several extensions are possible. Firstly, our model is simplified with two conductances, as a consequence periodic solutions of the deterministic model are not observed. Our coupling method should extend without difficulty to more elaborate Morris-Lecar systems, \red{as described in the introduction}. Secondly, we have used a standard mean-field limit as used for particles approximation in kinetic theory. The extension to more realistic mean-field limits leading to Integrate-and-Fire equations, with coupling between neurons through voltage discharges, is an interesting extension. The mathematical analysis of the underlying nonlinear kinetic equations faces difficulties due to boundary conditions that change type and, as here, to degeneracy of the Fokker-Planck equation. Existence of bounded \red{stationary} solutions has been developed in \cite{PeSa, DPSZ23}, numerical analysis in \cite{CCTao}. Also, the derivation from stochastic models of interacting individual neurons has attracted a lot of attention, see \cite{BoFaTa,DELARUE20152451} and references therein.

\subsection*{Acknowledgements}

The research of P.M. is supported by the French ANR grant SWIDIMS (ANR-20-CE40-0022). The research of B.P. has received support from ANR ChaMaNe No: ANR-19-CE40-0024.

\bibliographystyle{plain}
\bibliography{biblio}

\begin{thebibliography}{10}

\bibitem{BakryGentilLedoux}
Dominique Bakry, Ivan Gentil, and Michel Ledoux.
\newblock {\em Analysis and geometry of {M}arkov diffusion operators}, volume
  348 of {\em Grundlehren der Mathematischen Wissenschaften [Fundamental
  Principles of Mathematical Sciences]}.
\newblock Springer, Cham, 2014.

\bibitem{blaustein2023concentration}
Alain Blaustein and Francis Filbet.
\newblock Concentration phenomena in {F}itzhugh--{N}agumo equations: A
  mesoscopic approach.
\newblock {\em SIAM Journal on Mathematical Analysis}, 55(1):367--404, 2023.

\bibitem{BoFaTa}
Mireille Bossy, Olivier Faugeras, and Denis Talay.
\newblock Clarification and complement to ``{M}ean-field description and
  propagation of chaos in networks of {H}odgkin-{H}uxley and
  {F}itz{H}ugh-{N}agumo neurons''.
\newblock {\em J. Math. Neurosci.}, 5:Art. 19, 23, 2015.

\bibitem{BressloffBook}
Paul~C. Bressloff.
\newblock {\em Waves in neural media}.
\newblock Lecture Notes on Mathematical Modelling in the Life Sciences.
  Springer, New York, 2014.
\newblock From single neurons to neural fields.

\bibitem{burdzy2006synchronous}
Krzysztof Burdzy, Zhen-Qing Chen, and Peter Jones.
\newblock Synchronous couplings of reflected {B}rownian motions in smooth
  domains.
\newblock {\em Illinois Journal of Mathematics}, 50(1-4):189--268, 2006.

\bibitem{CCP_JMN11}
Mar\'{\i}a~J. C\'{a}ceres, Jos\'{e}~A. Carrillo, and Beno\^{i}t Perthame.
\newblock Analysis of nonlinear noisy integrate \& fire neuron models: blow-up
  and steady states.
\newblock {\em J. Math. Neurosci.}, 1:Art. 7, 33, 2011.

\bibitem{CCTao}
Maria~J. C\'aceres, Jos{\'e}~A. Carrillo, and Louis Tao.
\newblock A numerical solver for a nonlinear {F}okker-{P}lanck equation
  representation of neuronal network dynamics.
\newblock {\em J. Comp. Phys.}, 230:1084--1099, 2011.

\bibitem{CTRM06}
David Cai, Louis Tao, Aaditya~V. Rangan, and David~W. McLaughlin.
\newblock Kinetic theory for neuronal network dynamics.
\newblock {\em Commun. Math. Sci.}, 4(1):97--127, 2006.

\bibitem{CTSL}
David Cai, Louis Tao, Michael Shelley, and McLaughlin~David W.
\newblock An effective kinetic representation of fluctuation-driven neuronal
  networks with application to simple and complex cells in visual cortex.
\newblock {\em PNAS}, 101:7757--7762, 2004.

\bibitem{ColombaniLeBris}
Laetitia Colombani and Pierre~Le Bris.
\newblock Propagation of chaos in mean field networks of {FitzHugh}-{N}agumo
  neurons.
\newblock {\em Mathematical Neuroscience and Applications}, Volume 3, jun 2023.

\bibitem{cranston1991gradient}
Michael Cranston.
\newblock Gradient estimates on manifolds using coupling.
\newblock {\em Journal of functional analysis}, 99(1):110--124, 1991.

\bibitem{delarue2015global}
Fran{\c{c}}ois Delarue, James Inglis, Sylvain Rubenthaler, Etienne Tanr{\'e},
  et~al.
\newblock Global solvability of a networked integrate-and-fire model of
  {McKean}--{V}lasov type.
\newblock {\em The Annals of Applied Probability}, 25(4):2096--2133, 2015.

\bibitem{DELARUE20152451}
François Delarue, James Inglis, Sylvain Rubenthaler, and Etienne Tanré.
\newblock Particle systems with a singular mean-field self-excitation.
  application to neuronal networks.
\newblock {\em Stochastic Processes and their Applications}, 125(6):2451--2492,
  2015.

\bibitem{Djellout}
Hac{\`e}ne Djellout, Arnaud Guillin, and Liming Wu.
\newblock {Transportation cost-information inequalities and applications to
  random dynamical systems and diffusions}.
\newblock {\em The Annals of Probability}, 32(3B):2702 -- 2732, 2004.

\bibitem{dolbeault2015hypocoercivity}
Jean Dolbeault, Cl{\'e}ment Mouhot, and Christian Schmeiser.
\newblock Hypocoercivity for linear kinetic equations conserving mass.
\newblock {\em Transactions of the American Mathematical Society},
  367(6):3807--3828, 2015.

\bibitem{DPSZ23}
Xu'an Dou, Beno\^{i}t Perthame, Delphine Salort, and Zhennan Zhou.
\newblock Bounds and long term convergence for the voltage-conductance kinetic
  system arising in neuroscience.
\newblock {\em Discrete Contin. Dyn. Syst.}, 43(3-4):1366--1382, 2023.

\bibitem{duong2023reducing}
Manh~Hong Duong, Paul-Eric Chaudru~de Raynal, Pierre Monmarch{\'e}, Milica
  Toma{\v{s}}evic, and Julian Tugaut.
\newblock Reducing exit-times of diffusions with repulsive interactions.
\newblock {\em ESAIM: Probability and Statistics}, 2023.

\bibitem{EberleGuillinZimmer}
Andreas Eberle, Arnaud Guillin, and Raphael Zimmer.
\newblock Couplings and quantitative contraction rates for {L}angevin dynamics.
\newblock {\em The Annals of Probability}, 2017.

\bibitem{evans2021asymptotic}
Josephine~A. Evans and Havva Yolda{\c{s}}.
\newblock On the asymptotic behavior of a run and tumble equation for bacterial
  chemotaxis.
\newblock {\em SIAM J. Math. Anal.}, 55(6):7635--7664, 2023.

\bibitem{kendall1986nonnegative}
Wilfrid~S Kendall.
\newblock Nonnegative {R}icci curvature and the {B}rownian coupling property.
\newblock {\em Stochastics: An International Journal of Probability and
  Stochastic Processes}, 19(1-2):111--129, 1986.

\bibitem{monmarche2019generalized}
Pierre Monmarch{\'e}.
\newblock Generalized ${\Gamma}$ calculus and application to interacting
  particles on a graph.
\newblock {\em Potential Analysis}, 50(3):439--466, 2019.

\bibitem{M27}
Pierre Monmarch{\'e}.
\newblock Almost sure contraction for diffusions on $\mathbb{ R}^d$.
  {A}pplication to generalized {L}angevin diffusions.
\newblock {\em Stochastic Processes and their Applications}, 161:316--349,
  2023.

\bibitem{MorrisLecar}
Catherine Morris and Harold Lecar.
\newblock Voltage oscillations in the barnacle giant muscle fiber.
\newblock {\em Biophysical journal}, 35(1):193--213, 1981.

\bibitem{PeSa}
Beno{\^\i}t Perthame and Delphine Salort.
\newblock On a voltage-conductance kinetic system for integrate \& fire neural
  networks.
\newblock {\em Kinet. Relat. Models}, 6(4):841--864, 2013.

\bibitem{RKC}
Aaditya~V. Rangan, Gregor Kova{\v c}i{\v c}, and David Cai.
\newblock Kinetic theory for neuronal networks with fast and slow excitatory
  conductances driven by the same spike train.
\newblock {\em Phys. Rev. E (3)}, 77(4):041915, 13, 2008.

\bibitem{RevuzYor}
Daniel Revuz and Marc Yor.
\newblock {\em Continuous martingales and Brownian motion}, volume 293.
\newblock Springer Science \& Business Media, 2013.

\bibitem{saisho}
Yasumasa Saisho.
\newblock Stochastic differential equations for multi-dimensional domain with
  reflecting boundary.
\newblock {\em Probability Theory and Related Fields}, 74(3):455--477, 1987.

\bibitem{SaSm22}
Delphine Salort and Didier Smets.
\newblock Convergence towards equilibrium for a model with partial diffusion.
\newblock {\em HAL preprint hal-03845918}, 2022.

\bibitem{villani2006hypocoercivity}
C\'{e}dric Villani.
\newblock Hypocoercivity.
\newblock {\em Mem. Amer. Math. Soc.}, 202(950):iv+141, 2009.

\bibitem{Villani}
C{\'e}dric Villani.
\newblock {\em Optimal transport, old and new}, volume 338 of {\em Grundlehren
  der Mathematischen Wissenschaften [Fundamental Principles of Mathematical
  Sciences]}.
\newblock Springer-Verlag, Berlin, 2009.

\bibitem{wang1994application}
Feng-Yu Wang.
\newblock Application of coupling methods to the {N}eumann eigenvalue problem.
\newblock {\em Probability Theory and Related Fields}, 98:299--306, 1994.

\end{thebibliography}
\end{document}